\documentclass[10pt,leqno]{article} %%%%leqno es para numeraci\'{o}n de
   \usepackage[centertags]{amsmath}
   \usepackage{amsfonts}
   \usepackage{amsmath}
   \usepackage{amssymb}
   \usepackage{amsthm}
   \usepackage{newlfont}
   \usepackage[latin1]{inputenc}%%%%para que reconozca las tildes.
\allowdisplaybreaks[3]
\theoremstyle{plain}
\newtheorem{theorem}{Theorem}[section]
\newtheorem{lemma}[theorem]{Lemma}

\theoremstyle{definition}
\newtheorem{definition}[theorem]{Definition}

\theoremstyle{remark}

\newtheorem*{remark*}{Remark}

\numberwithin{equation}{section}

\newcommand\D{{\mathcal D}}
\newcommand\A{{\mathcal A}}

\newcommand\cR{{\mathcal R}}

\newcommand\RR{{\mathbb R}}

\newcommand\NN{{\mathbb N}}
\newcommand\PP{{\mathbb P}}

\newcommand\awr{\operatorname{\mbox{$\alpha$}-wr}}

\newcommand\rank{\operatorname{rank}}

\newcommand\diagonal{\operatorname{diag}}

   \title{Differential equations for discrete Laguerre-Sobolev orthogonal polynomials
  \footnote{Partially supported by MTM2012-36732-C03-03 (Ministerio de Economía y Competitividad),
FQM-262, FQM-4643, FQM-7276 (Junta de Andalucía) and Feder Funds (European
Union). The second author is also grateful to FOS studio in Panam\'a, for letting him use their facilities during the summer of 2013.}}
   \author{Antonio J. Durán and Manuel D. de la Iglesia$^{\dagger}$ \\
   \footnotesize $\dagger$  \footnotesize
    \  Departamento de An\'{a}lisis Matem\'{a}tico.
   Universidad de Sevilla \\
   \footnotesize Apdo (P. O. BOX) 1160. 41080 Sevilla. Spain.
   duran@us.es, mdi29@us.es \\}
\date{}

\begin{document}
   \maketitle

   \begin{abstract}
The aim of this paper is to study differential properties of orthogonal polynomials with respect to a discrete Laguerre-Sobolev bilinear form
with mass point at zero. In particular we construct the orthogonal polynomials using certain Casorati determinants. Using this construction, we prove that they are eigenfunctions of a differential operator (which will be explicitly constructed). Moreover, the order of this differential operator is explicitly computed in terms of the matrix which defines the discrete Laguerre-Sobolev bilinear form.
\end{abstract}

\section{Introduction and results}
The  issue of orthogonal polynomials (with respect to a positive
measure) which are also common eigenfunctions of a second order
differential operator goes back at least for two centuries, when
Legendre introduced the first family of what we call today
classical orthogonal polynomials. As S. Bochner established in
1929 \cite{B}, there are only three families of classical
orthogonal polynomials: Hermite, Laguerre and Jacobi (and Bessel
polynomials if signed measures are considered).

H.L. Krall raised in 1939 (\cite{Kr1,Kr2}) the problem of finding
orthogonal polynomials which are also common eigenfunctions of a
higher order differential operator with polynomial coefficients.
He obtained a complete classification for the case of a
differential operator of order four (\cite{Kr2}). Besides the
classical families of Hermite, Laguerre and Jacobi (satisfying
second order differential equations), he  found three other
families of orthogonal polynomials which are also eigenfunctions of a fourth order
differential operator. One of them is orthogonal with
respect to a positive measure which consists of a Laguerre weight
together with a Dirac delta at the end point of its interval of
orthogonality: $e^{-x}+M_0\delta _0$.

Forty years later, L.L. Littlejohn (\cite{L1,L2}) discovered new
families satisfying sixth and eighth order differential equations, respectively. They are
orthogonal with respect to
\begin{equation}\label{krr1}
x^\alpha e^{-x}+M_0\delta _0,\quad x>0,
\end{equation}
with $\alpha =1, 2$, respectively. The general result for $\alpha$ a nonnegative
integer was proved by J. Koekoek and R. Koekoek who showed in 1991
that orthogonal polynomials with respect to \eqref{krr1} are also eigenfunctions
of an infinite order differential operator, except for nonnegative
integer values of $\alpha$ for which the order reduces to $2\alpha
+ 4$ (\cite{koekoe}).

Some years later discrete Laguerre-Sobolev orthogonal polynomials
which are also common eigenfunctions of a higher order
differential operator entered into the picture. R. Koekoek and H.
G. Meijer \cite{KM} introduced orthogonal polynomials with respect
to the discrete Laguerre-Sobolev inner product
$$
\langle p,q\rangle=\int_0^\infty
p(x)q(x)x^{\alpha}e^{-x}dx+M_0p(0)q(0)+M_1p'(0)q'(0),\quad M_0\ge 0, M_1>0,
$$
and later on R. Koekoek \cite{koea,koeb} found that for  $\alpha
=0,1,2$, those orthogonal polynomials are also eigenfunctions of a
differential operator with polynomial coefficients of order
$2\alpha +8$ when $M_0=0$ and $4\alpha +10$ for $M_0>0$. This result
was soon extended for nonnegative integers $\alpha $ by J.
Koekoek, R. Koekoek and H. Bavinck \cite{KKB}. Using a different
approach, P. Iliev \cite{I} has recently extended these results
for a Laguerre-Sobolev inner product of the form
$$
\langle p,q\rangle=\int_0^\infty p(x)q(x)x^{\alpha-2}e^{-x}dx+(p(0),p'(0))\begin{pmatrix}
M_{0,0}&M_{0,1}\\M_{0,1}&M_{1,1}\end{pmatrix}
\begin{pmatrix} q(0)\\q'(0) \end{pmatrix}.
$$
(For other related papers see \cite{JKLL1} and \cite{JKLL2}).

\medskip

For $\alpha \not =-1,-2,\ldots,$  denote $\mu_{\alpha}(x)$ the orthogonalizing weight for the
Laguerre polynomials. Only when $\alpha >-1$, $\mu_{\alpha }(x)$, $x>0$, is positive, and then
\begin{equation}\label{wL}
\mu_\alpha(x) =x^\alpha e^{-x},\quad x>0.
\end{equation}
Let $M$ be a $m\times m$ matrix. The purpose of this paper is to prove in a constructive way that
if $\alpha$ and $m$ are positive integers with $\alpha \ge m$, then
the orthogonal polynomials with respect to a discrete Laguerre-Sobolev bilinear form
$$
\langle p,q\rangle=\int_0^\infty
p(x)q(x)\mu_{\alpha-m}(x)dx+(p(0),\ldots,
p^{(m-1)}(0))M\begin{pmatrix} q(0)\\\vdots \\q^{(m-1)}(0)
\end{pmatrix},
$$
are eigenfunctions of a differential operator with polynomial coefficients.

To do that we introduce  the functions $\mathcal R_l$, $l=1,\ldots
, m$, defined by
\begin{equation}\label{iRdSip}
\mathcal R_l(x)=\frac{\Gamma (\alpha
-m+l)}{(m-l)!}(x+1)_{m-l}+(l-1)!\frac{\Gamma(\alpha
+1+x)}{\Gamma(1+x)}\sum_{i=0}^{m-1}\frac{(-1)^iM_{l-1,i}}{\Gamma(\alpha
+i+1)}(x-i+1)_i,
\end{equation}
where $(x)_n$ denotes the Pochhammer symbol.

Using a general result for discrete Sobolev bilinear forms
(Theorem \ref{LemLagSob} in the Preliminaries), we first
characterize the existence of (left) orthogonal polynomials with
respect to the Laguerre-Sobolev bilinear form above using the
Casorati determinant defined by the functions $\mathcal R_l$,
$l=1,\ldots , m$. Moreover, we find a close expression for these
orthogonal polynomials in terms of the Laguerre polynomials
$(L_n^\alpha)_n$ and the functions $\mathcal R_l$.

\begin{theorem}\label{mainth1}
For $m\geq1$, let $M=(M_{i,j})_{i,j=0}^{m-1}$ be a $m\times m$
matrix. For $\alpha \not =m-1, m-2,\ldots $, consider the discrete
Laguerre-Sobolev bilinear form defined by
\begin{equation}\label{idslip}
\langle p,q\rangle=\int_0^\infty
p(x)q(x)\mu_{\alpha-m}(x)dx+\mathbb{P}(0)M\mathbb{Q}(0)^T,
\end{equation}
where $\mathbb{P}(0)=\left(p(0),p'(0),\ldots,p^{(m-1)}(0)\right)$
and $\mu_{\alpha}$ is the orthogonalizing weight for the Laguerre polynomials $(L_n^\alpha )_n$. Then the following conditions are equivalent
\begin{enumerate}
\item The discrete Laguerre-Sobolev bilinear form (\ref{idslip}) has a sequence $(q_n)_n$ of (left) orthogonal polynomials.
\item The $m\times m$ Casorati determinant
\begin{equation}\label{casdet}
\Omega (n)=\det (\mathcal R_i(n-j))_{i,j=1}^m,
\end{equation}
where $\mathcal R_l, l=1,\ldots,m,$ are defined by \eqref{iRdSip}, does not vanish for $n\geq0$.
\end{enumerate}
Moreover, if one of these properties holds, the polynomials defined
by
\begin{equation}\label{iquss}
q_n(x)=\begin{vmatrix}
               L_n^\alpha (x) & L^\alpha_{n-1}(x) & \cdots & L^\alpha_{n-m}(x) \\
\mathcal R_1(n) &\mathcal R_1(n-1) & \cdots & \mathcal R_1(n-m) \\
               \vdots & \vdots & \ddots & \vdots \\
               \mathcal R_m(n) &  \mathcal R_m(n-1) & \cdots & \mathcal R_m(n-m)
             \end{vmatrix},\quad n\geq0,
\end{equation}
are orthogonal with respect to (\ref{idslip}) (as usual for $n<0$ we take $L_n^\alpha =0$).
\end{theorem}

We find the differential properties of the orthogonal polynomials
$(q_n)_n$ (\ref{iquss}) by using the concept of $\D$-operator.
This is an abstract concept introduced by one of the authors in
\cite{du1} which has shown to be very useful to generate
orthogonal polynomials which are also eigenfunctions of
differential, difference or $q$-difference operators (see
\cite{AD,du1,ddI}). The basic facts about $\D$-operators will be
recalled in Section 3. Using the general theory of $\D$-operators
and the expression (\ref{iquss}) for the orthogonal polynomials
$(q_n)_n$, we construct an algebra of differential operators for
which they are eigenfunctions. An important issue will be
the explicit calculation of the order of these operators in terms
of the matrix $M$ which defines the
 discrete Laguerre-Sobolev bilinear form (\ref{idslip}). This will raise  the concept of \emph{$\alpha$-weighted rank} associated to a  matrix $M$ defined as follows.

\begin{definition}\label{wrM}
Let $M$ be a $m\times m$ matrix. Write $c_1,\ldots , c_m$, for the columns of $M$ and define the numbers $n_j$, $j=1,\ldots , m,$ by
$$
n_1=\begin{cases} \alpha +m-1,& \mbox{if $c_m\not =0$,}\\ 0,&\mbox{if $c_m=0$;}\end{cases}
$$
and for $j=2,\ldots , m,$
\begin{equation}\label{nj}
n_j=\begin{cases} \alpha +m-j,& \mbox{if $c_{m-j+1}\not \in <c_{m-j+2},\ldots , c_m>$,}\\ 0,&\mbox{if $c_{m-j+1} \in <c_{m-j+2},\ldots , c_m>$.}\end{cases}
\end{equation}
Denote by $\tilde M$ the matrix whose columns are $c_i$, $i\in \{ j:n_{m-j+1}\not=0\}$ (i.e., the columns of $\tilde M$ are (from right to left) those columns $c_i$ of $M$ such that $c_i\not \in <c_{i+1},\ldots , c_m>$). Write $f_1,\ldots , f_m,$ for  the rows of $\tilde M$. We define the numbers $m_j$, $j=1,\ldots , m-1,$ by
\begin{equation}\label{mj}
m_j=\begin{cases} m-j,& \mbox{if $f_{j} \in <f_{j+1},\ldots , f_m>$,}\\ 0,&\mbox{if $f_{j}\not \in <f_{j+1},\ldots , f_m>$.}\end{cases}
\end{equation}
The \emph{$\alpha$-weighted rank} of the matrix $M$, $\awr (M)$ in short, is then defined by
$$
\awr(M)=\sum_{j=1}^m n_j+\sum_{j=1}^{m-1}m_j-\frac{m(m-1)}{2}.
$$
\end{definition}

We are now ready to establish in detail the differential
properties of the orthogonal polynomials $(q_n)_n$ (\ref{iquss}).

\begin{theorem}\label{mainth2}
Assume that any of the two equivalent properties (1) and (2) in
Theorem \ref{mainth1} hold, and assume, in addition, that $\alpha$
is a positive integer with $\alpha \ge m$. Then, for any polynomial
$S$, there exists a differential operator $D_S$ (which will be
explicitly constructed) of order $2(\deg S+\awr (M)+1)$ with
respect to which the orthogonal polynomials $(q_n)_n$
(\ref{iquss}) are eigenfunctions ($\awr (M)$ is the
$\alpha$-weighted rank of the matrix $M$ defined above). Moreover,
up to an additive constant, the corresponding eigenvalues
$(\lambda_n)_n$ of $D_S$ are $\lambda _n=P_S(n)$, where $P_S$ is
the polynomial defined by the first order difference equation
$$P_S(x)-P_S(x-1)=S(x)\Omega (x),$$ and $\Omega$ is the Casorati
determinant (\ref{casdet}). In particular, the minimal order of the differential operators having the orthogonal polynomials $(q_n)_n$ as eigenfunctions is at most $2(\awr (M)+1)$.
\end{theorem}

When $M=\begin{pmatrix} M_{0,0}&M_{0,1}\\M_{0,1}&M_{1,1}\end{pmatrix}$, we get
$$
\awr (M)=\begin{cases}
\alpha -1 , &\mbox{for $M_{0,0}\not =0, M_{1,1}=0, M_{0,1}=0$},\\
\alpha +1 , &\mbox{for $M_{0,0}=0, M_{1,1}\not=0, M_{0,1}=0$},\\
2\alpha , &\mbox{for $M_{0,0}\not =0, M_{1,1}\not=0, M_{0,1}=0$},
\end{cases}
$$
and we recover Koekoeks' (\cite{koekoe}), Koekoeks' and Bavinck (\cite{KKB}), and Iliev results (\cite{I}).

\section{Preliminaries. Discrete Sobolev inner products}
We say that a sequence of polynomials $(q_n)_n$, with $q_n$ of
degree $n$, $n\ge 0$, is (left) orthogonal with respect to a
bilinear form $B$ (not necessarily symmetric) defined in the
linear space of real polynomials if $B(q_n,q)=0$ for all
polynomials $q$ with degree of $q$ less that $n$ and $B(q_n,q_n)\not
=0$. It is clear from the definition that (left) orthogonal polynomials with respect to a bilinear form, if they exist, are unique up to multiplication by nonnegative constants. Given a measure $\mu$ (positive or not), with finite moments of any order, we consider the bilinear form
$B_\mu (p,q)=\int pqd\mu$. We then say that a sequence of polynomials $(q_n)_n$, with $q_n$ of
degree $n$, $n\ge 0$, is orthogonal with respect to the measure $\mu$ if it is orthogonal with respect to the bilinear form $B_\mu$.

We will use the following Lemma to construct (left) orthogonal
polynomials with respect to a discrete Sobolev bilinear form. We have not found in the literature the formula (\ref{quss}): it gives the discrete Sobolev polynomials in terms of $m$ consecutive orthogonal polynomials with respect to the measure $(x-\lambda)^m\nu$, when it is more usual to use
the orthogonal polynomials with respect to the measure $\nu$, or other variants (see \cite{AMRR, MSR}).

\begin{lemma}\label{LemLagSob}
For $m\geq1$, let $M=(M_{i,j})_{i,j=0}^{m-1}$ be a $m\times m$ matrix. For a given measure $\nu$ and for certain real number $\lambda$ consider the discrete Sobolev bilinear form defined by
\begin{equation}\label{DSip}
\langle p,q\rangle=\int p(x)q(x)d\nu(x)+\mathbb{P}(\lambda)M\mathbb{Q}(\lambda)^T,
\end{equation}
where
$\mathbb{P}(\lambda)=\left(p(\lambda),p'(\lambda),\ldots,p^{(m-1)}(\lambda)\right)$.
Assume that the measure $(x-\lambda)^m\nu$ has a sequence
$(p_n)_n$ of orthogonal polynomials, and write $w_{n,i}=\int
(x-\lambda)^ip_{n}d\nu $, $i=1,\ldots,m$. For $l=1,\ldots,m,$
define the sequences $(R_l(n))_n$
\begin{equation}\label{RdSip}
R_l(n)=w_{n,l-1}+(l-1)!\sum_{i=0}^{m-1}M_{l-1,i}p_n^{(i)}(\lambda).
\end{equation}
Then the following conditions are equivalent
\begin{enumerate}
\item The discrete Sobolev bilinear form (\ref{DSip}) has a sequence $(q_n)_n$ of (left) orthogonal polynomials (in particular $\deg(q_n)=n$).
\item The $m\times m$ Casorati determinant $\Omega (n)=\det (R_i(n-j))_{i,j=1}^m$ does not vanish for $n\geq0$.
\end{enumerate}
Moreover, if one of these properties holds, the polynomial defined
by
\begin{equation}\label{quss}
q_n(x)=\begin{vmatrix}
               p_n(x) & p_{n-1}(x) & \cdots & p_{n-m}(x) \\
R_1(n) &R_1(n-1) & \cdots & R_1(n-m) \\
               \vdots & \vdots & \ddots & \vdots \\
               R_m(n) &  R_m(n-1) & \cdots & R_m(n-m)
             \end{vmatrix}
\end{equation}
has degree $n$, $n\ge 0$, and the sequence $(q_n)_n$ is (left)
orthogonal with respect to (\ref{DSip}) (for $n<0$ we take $p_n =0$).
\end{lemma}

\begin{proof}
We first prove (1) $\Rightarrow $ (2).

We can assume that both polynomials  $q_n$ and $p_n$ have equal leading coefficient. Hence we can write
$$
q_n(x)=p_n(x)+\sum_{j=1}^n\beta_{n,j}p_{n-j}(x).
$$
The definition of the discrete Sobolev bilinear form (\ref{DSip}) gives then for $m<j\le n$,
\begin{align*}
\beta_{n,j}\int p_{n-j}^2(x-\lambda)^md\nu &=\int q_n(x)p_{n-j}(x)(x-\lambda)^md\nu\\
&=\langle q_n,(x-\lambda)^m p_{n-j}\rangle =0.
\end{align*}
Since we are assuming that $(p_n)_n$ are orthogonal with respect to $(x-\lambda)^m\nu$, we have $\int p_{n-j}^2(x-\lambda)^md\nu \not=0$. Then, for $n\ge 0$,
$$
q_n(x)=\sum_{j=0}^m\beta_{n,j}p_{n-j}(x),
$$
where we set $\beta_{n,0}=1$. Using this identity, we get
\begin{align*}
\langle q_n,(x-\lambda)^l\rangle &=\int (x-\lambda)^lq_n(x)d\nu +l!\sum_{i=0}^{m-1}q_n^{(i)}(\lambda)M_{l,i}\\
&=\sum_{j=0}^m\beta_{n,j}\int (x-\lambda)^lp_{n-j}(x)d\nu +l!\sum_{i=0}^{m-1}M_{l,i}\sum_{j=0}^m\beta_{n,j}p_{n-j}^{(i)}(\lambda)\\
&=\sum_{j=0}^m\beta_{n,j}w_{n-j,l} +l!\sum_{j=0}^m\beta_{n,j}\sum_{i=0}^{m-1}M_{l,i}p_{n-j}^{(i)}(\lambda)\\
&=R_{l+1}(n)+\sum_{j=1}^m\beta_{n,j}R_{l+1}(n-j).
\end{align*}
Assume now $n\ge m$. For $l=0,\ldots , m-1$, we then get
$$
0=\langle q_n,(x-\lambda)^l\rangle =R_{l+1}(n)+\sum_{j=1}^m\beta_{n,j}R_{l+1}(n-j).
$$
This shows that the linear system
\begin{equation}\label{lins}
\sum_{j=1}^m\phi_{n,j}R_l(n-j)=-R_l(n),\quad l=1,\ldots, m,
\end{equation}
has at least a solution $\phi_{n,j}$, $j=1,\ldots, m$; more precisely this solution is given by
$\phi_{n,j}=\beta_{n,j}$, $j=1,\ldots, m$. We now prove that this is the unique solution of the linear system (\ref{lins}).

Take any other solution $\phi_{n,j}$, $j=1,\ldots, m$, of the linear system (\ref{lins}) and define the polynomial of degree $n$
$$
\tilde q_n(x)=\sum_{j=0}^m\phi_{n,j}p_{n-j}(x),
$$
where we write $\phi_{n,0}=1$. Proceeding as before we get for $l=0,\ldots, m-1,$
$$
\langle \tilde q_n,(x-\lambda)^l\rangle =R_{l+1}(n)+\sum_{j=1}^m\phi_{n,j}R_{l+1}(n-j)=0.
$$
On the other hand, for $l=m,\ldots , n-1$, we have
$$
\langle \tilde q_n,(x-\lambda)^l\rangle =\int \tilde q_n(x)(x-\lambda)^ld\nu=\sum _{j=0}^m\phi _{n,j}\int p_{n-j}(x)(x-\lambda)^{l-m}(x-\lambda)^md\nu=0.
$$
This is saying that $\tilde q_n$ is orthogonal to any polynomial of degree less than $n$ with respect to the discrete Sobolev bilinear form (\ref{DSip}).
Since $\tilde q_n$ and $q_n$ have degree $n$ and equal leading coefficient they must be equal. Hence $\phi_{n,j}=\beta_{n,j}$, $j=1,\ldots ,m$. Since
the linear system (\ref{lins}) has only one solution, we deduce that $\Omega(n)\not =0, n\geq0$.

If $0\le n<m$ we can proceed in a similar way.

We now prove that (2) $\Rightarrow $ (1). Since $\Omega(n)\not =0$, $n\ge 0$, the linear system (\ref{lins}) has a unique solution which we call
$\beta_{n,j}$. Define now the polynomials
$$
q_n(x)=\sum_{j=0}^m\beta_{n,j}p_{n-j}(x),
$$
where we again write $\beta_{n,0}=1$.
Assume first that $n\ge m$. Proceeding as before, we can prove that $\langle q_n,(x-\lambda)^l\rangle=0$, $l=0,\ldots , n-1$. It is then enough to prove that $\langle q_n,(x-\lambda)^n\rangle\not =0$. But
\begin{align*}
\langle q_n,(x-\lambda)^n\rangle &=\int
q_n(x)(x-\lambda)^nd\nu=\sum _{j=0}^m\beta_{n,j}\int
p_{n-j}(x)(x-\lambda)^{n-m}(x-\lambda)^md\nu\\ &=\beta_{n,m}\int
p_{n-m}(x)(x-\lambda)^{n-m}(x-\lambda)^md\nu\not =0,
\end{align*}
since $\beta_{n,m}=\Omega(n+1)$.

For $0\le n<m$, we can proceed analogously.

\end{proof}

The following technical result will be used later in Lemma \ref{Lemawr}.

\begin{lemma}\label{cdm}
Consider the vectors $q_1,\ldots , q_r\in\mathbb{R}^m$. Denote by $e_i$, $i=1,\ldots , m$, the canonical basis in $\mathbb{R}^m$, that is $e_{i,j}=\delta_{i,j}$. Assume that
\begin{enumerate}
\item the vectors $q_i$, $i=1,\ldots , r$, are unitary and orthogonal;
\item for certain $i_0$, $1\le i_0<m$, $e_1,\ldots, e_{i_0}\in <q_1,\ldots , q_r>;$
\item $e_{i_0+1}\not \in  <q_1,\ldots , q_r>$.
\end{enumerate}
Then there exists an unitary vector $q_{r+1}$ orthogonal to $q_1,\ldots , q_r$ satisfying that $q_{r+1}\in <q_1,\ldots , q_r,e_{i_0+1}>$ and $e_{i_0+1}\in <q_1,\ldots , q_r,q_{r+1}>.$
\end{lemma}

\begin{proof}
Indeed, since $e_{i_0+1}\not \in  <q_1,\ldots , q_r>$, we can write $e_{i_0+1}=u+v$, where $u\in <q_1,\ldots , q_r>$, $v\in <q_1,\ldots , q_r>^\bot$, $v\not =0$. Define $q_{r+1}=v/\Vert v\Vert $. Then $q_{r+1}$ is unitary and orthogonal to $q_1,\ldots , q_r$.  By construction, we  see that
$q_{r+1}\in <q_1,\ldots , q_r,e_{i_0+1}>$ and $e_{i_0+1}\in <q_1,\ldots , q_r,q_{r+1}>.$
\end{proof}

To make the reading of this paper easier, we include here some known formulas for Laguerre polynomials which we will need later.

\subsection{Laguerre polynomials.}\label{LGP} For $\alpha\in\mathbb{R}$, we use the standard definition of the Laguerre polynomials $(L_n^{\alpha})_n$ (see \cite{EMO}, pp. 188--192)
\begin{equation}\label{deflap}
L_n^{\alpha}(x)=\sum_{j=0}^n\frac{(-x)^j}{j!}\binom{n+\alpha}{n-j}.
\end{equation}
They satisfy a three-term recurrence formula ($L_{-1}^{\alpha}=0$)
\begin{equation*}\label{ttrrL}
xL_n^{\alpha}=-(n+1)L_{n+1}^{\alpha}+(2n+\alpha+1)L_n^{\alpha}-(n+\alpha)L_{n-1}^{\alpha}.
\end{equation*}
Hence, for $\alpha\neq-1,-2,\ldots$, they are orthogonal with respect to a measure $\mu_{\alpha}=\mu_{\alpha}(x)dx$. This measure is positive
only when $\alpha>-1$ and then $\mu_{\alpha}$ is defined by \eqref{wL}.

The Laguerre polynomials are eigenfunctions of the following second-order differential operator
\begin{equation}\label{DopL}
D_{\alpha}=-x\left(\frac{d}{dx}\right)^2-(\alpha+1-x)\frac{d}{dx},\quad D_{\alpha}(L_n^{\alpha})=nL_n^{\alpha},\quad n\geq0.
\end{equation}
We will also use the following formulas
\begin{equation}\label{Lagder}
    \left(L_n^{\alpha}\right)'=-L_{n-1}^{\alpha+1},
\end{equation}
\begin{equation}\label{Lagab}
   L_n^{\alpha}=\sum_{j=0}^n\frac{(\alpha-\beta)_j}{j!}L_{n-j}^{\beta}.
\end{equation}

\section{$\D$-operators}

The concept of $\D$-operator was introduced by one of the authors in \cite{du1}. In \cite{du1}, \cite{AD} and \cite{ddI} it has been showed that $\D$-operators turn out to be an extremely useful tool of a method
to generate families of polynomials which are eigenfunctions of higher order differential, difference or $q$-difference operators. The purpose of this section is to remind that method and its main ingredient: $\D$-operators.

The starting point is a sequence of polynomials $(p_n)_n$, $\deg p_n=n$, and an algebra of operators $\A $ acting in the linear space of
polynomials $\mathbb{P}$. In this paper, we consider the algebra $\A$ formed by all finite order differential operators $\sum_{j=0}^sf_j(d/dx)^j$ where $f_j\in\PP$ is a polynomial with degree at most $j$, $j=0,\ldots, s$:
\begin{equation}\label{algdiffc}
\A =\left\{ \sum_{j=0}^sf_j\left(\frac{d}{dx}\right)^j : f_j\in \PP, \deg(f_j)\le j, j=0,\ldots,s, s\in \NN \right\}.
\end{equation}
If $f_s\neq0$ we say that the \emph{order} of such differential operator is $s$.

In addition, we assume that the polynomials $p_n$, $n\ge 0$, are eigenfunctions of certain operator $D_p\in \A$. We write $(\theta_n)_n$ for the corresponding eigenvalues, so that $D_p(p_n)=\theta_np_n$, $n\ge 0$. In this paper we only consider the case when the sequence of eigenvalues $(\theta _n)_n$ is linear in $n$ (for $\D$-operators associated to polynomials $(p_n)_n$ for which the sequence of eigenvalues $(\theta _n)_n$ is not linear in $n$ see \cite{du1}).

Given  a sequence of numbers $(\varepsilon_n)_n$, a $\D$-operator associated to the algebra $\A$ and the sequence of polynomials
$(p_n)_n$ is defined as follows.
We first consider  the operator $\D :\PP \to \PP $ defined by linearity
from
\begin{equation}\label{defTo}
\D (p_n)=\sum _{j=1}^n (-1)^{j+1}\varepsilon _n\cdots \varepsilon _{n-j}p_{n-j},\quad n\ge 0.
\end{equation}
We then say that $\D$ is a $\D$-operator if $\D\in \A$.

Let us now provide an example of $\D$-operator for the Laguerre polynomials.

\begin{lemma}\label{lTlag} (Lemma A.1 of \cite{du1})
For $\alpha \in \RR $, let $p_{n}=L_{n}^{\alpha }$, $n\ge 0$, be the Laguerre polynomials given by (\ref{deflap}).
Then the operator $\D$ defined by (\ref{defTo}) from the sequence $\varepsilon _n=-1$, $n\ge 0$,
is a $\D$-operator for the Laguerre polynomials and the algebra $\A$  (\ref{algdiffc}).
More precisely  $\D=d/dx$.
\end{lemma}

We now show how to use $\D$-operators to construct new sequences of polynomials $(q_n)_n$ such that there exists an operator $D_q\in \A$ for which they are eigenfunctions. To do that we take $m$ arbitrary polynomials $R_1, R_2, \ldots, R_m$ and consider the $m\times m$  Casorati determinant defined by
\begin{equation}\label{casd1}
\Omega (x)=\det \left(R_l(x-j)\right)_{l,j=1}^m.
\end{equation}

The details of our method are included in the following Theorem (which it is the particular case of Theorem 3.2 in \cite{ddI} for the case where we only use one $\mathcal D$-operator for $(p_n)_n$ and $\mathcal A$).

\begin{theorem}\label{Teor1} Let $\A$ and $(p_n)_n$ be, respectively, an algebra of operators acting in the linear space of polynomials, and a sequence of polynomials with $\deg p_n=n$. We assume that $(p_n)_n$ are eigenfunctions of an operator $D_p\in \A$ with eigenvalues equal to $n$, that is, $D_p(p_n)=np_n$, $n\ge 0$. We also have a sequence of numbers $(\varepsilon_{n} )_n$, which defines a $\D$-operator $\D$  for $(p_n)_n$ and $\A$ (see \eqref{defTo})). We write $\xi_{n,i}$, $n\ge 0, 0\le i\le n$,  for the auxiliary numbers defined by
\begin{equation*}\label{defxi}
\xi_{n,i}=\prod_{j=0}^{i-1}\varepsilon_{n-j}, \quad 1\le i\le n,\quad \quad \xi_{n,0}=1.
\end{equation*}

Let $R_1, R_2, \ldots, R_m$ be $m$ arbitrary polynomials satisfying that $\Omega (n)\not =0$, $n\ge 0$, where $\Omega $ is the Casorati determinant defined by \eqref{casd1}.

Consider the sequence of polynomials $(q_n)_n$ defined by
\begin{equation}\label{qus}
q_n(x)=\begin{vmatrix}
               p_n(x) & -\xi_{n,1}p_{n-1}(x) & \cdots & (-1)^m\xi_{n,m}p_{n-m}(x) \\
R_1(n) &R_1(n-1) & \cdots & R_1(n-m) \\
               \vdots & \vdots & \ddots & \vdots \\
               R_m(n) &  R_m(n-1) & \cdots & R_m(n-m)
             \end{vmatrix}.
\end{equation}
For a polynomial $S$ and $h=1,\ldots,m$, we define the polynomials $M_h(x)$ by
\begin{align}\label{emeiexp}
\nonumber M_h(x)&=\sum_{j=1}^m(-1)^{h+j}S(x+j)\times\\
&\qquad\times\det\left(R_l(x-r)\right)_{\scriptsize \left\{\begin{array}{l}
                                                           l\in\{1,2,\ldots,m\}\setminus\{h\} \\
                                                            r\in\{-j+1,-j+2,\ldots,m-j\}\setminus\{0\}
                                                          \end{array}\right\}}.
\end{align}
Then there exists an operator $D_{q,S}\in \A$ such that
$$
D_{q,S}(q_n)=\lambda_nq_n,\quad n\ge 0.
$$
Moreover, an explicit expression of this operator can be displayed. Indeed, write  $P_S$ for the polynomial defined by
\begin{equation}\label{Pgs}
P_S(x)-P_S(x-1)=S(x)\Omega (x).
\end{equation}
Then the operator $D_{q,S}$ is defined by
\begin{equation}\label{Dq}
D_{q,S}=P_S(D_p)+\sum_{h=1}^mM_h(D_p)\D R_h(D_p),
\end{equation}
where $D_p\in \A$ is the operator for which the polynomials $(p_n)_n$ are eigenfunctions. Moreover
$\lambda_n=P_S(n)$.

\end{theorem}

\bigskip
When $D_p$ is a differential operator, the polynomial $P_S$ (see \eqref{Pgs}) will give the order of the differential operator $D_{q,S}$ \eqref{Dq}.
This is a consequence of the following three lemmas. The first lemma gives the degrees of the polynomials $\Omega $ and $M_h$ and it is a particular case of Lemmas 3.4 and 3.5 in \cite{ddI} (when $m_2=0$).

\begin{lemma}\label{lgp1} For a nonnegative integer $m$,  let $R_1, R_2, \ldots, R_m,$ be non null polynomials satisfying that $\deg R_i\not =\deg R_{j}$, for $i\not =j$. Let $l_1, \ldots , l_m,$ be numbers satisfying that $l_i\not =l_j$, $i\not =j$. Write also $U_j(x)$, $j=1,\ldots , m+1$, for the matrix with row $(U_j(x))_l$ equal to
$$
(U_j(x))_l=(R_l(x-r))_{{\scriptsize
r=1-j,2-j,\ldots,m+1-j, r\not =0}}.
$$
We then have
\begin{align}
\label{fL1} \deg \left(\det (R_i(x-l_j))_{i,j=1}^m\right) &=\left(\sum_{i=1}^m\deg R_i\right) -\binom{m}{2},\\
\label{fL2} \deg \left(\sum _{j=1}^{m+1} (-1)^{j+1}\det U_j(x)\right) &\le \left(\sum_{i=1}^m\deg R_i\right)-\binom{m+1}{2}.
\end{align}
\end{lemma}

In order to enunciate the next lemma we need some notation.
Given $m$ arbitrary polynomials $R_1,\ldots , R_m$, we will denote by $\cR $ the $m$-tuple of polynomials $(R_1,\ldots, R_m)$. The $m$-tuple formed by interchanging the polynomials $R_i$ and $R_j$ in $\cR$ is denoted by $\cR_{i\leftrightarrow j}$; the $m$-tuple formed by changing the polynomial $R_i$ to $aR_i+bR_j$ in $\cR$, where $a$ and $b$ are real numbers, is denoted by $\cR_{i\leftrightarrow ai+bj}$; and the $m$-tuple formed by removing the polynomial $R_i$ in $\cR$ is denoted by $\cR_{\{ i\} }$.

\begin{lemma}
Given $m$ arbitrary polynomials $R_1,\ldots , R_m$, we form the $m$-tuple of polynomials
$\cR =(R_1,\ldots, R_m)$ and consider the  operator $D_{q,S}=D_{q,S}(\cR )$ \eqref{Dq}. Then, for any numbers $a, b\in\mathbb{R}$ we have
\begin{align} \label{Dqinv1}
D_{q,S}(\cR )&=-D_{q,S}(\cR_{i\leftrightarrow j}),\\\label{Dqinv2}
D_{q,S}(\cR_{i\leftrightarrow ai+bj})&=aD_{q,S}(\cR).
\end{align}
\end{lemma}
\begin{proof}
The dependence of $D_{q,S}$ in terms of $\cR =(R_1,\ldots , R_m)$ (see (\ref{Dq})) comes in terms of the polynomials $P_S$ and $M_h, h=1,\ldots,m$. It is easy to see, using $P_S(x)-P_S(x-1)=S(x)\Omega(x)$, that $P_S=P_S(\cR )$ satisfy the properties \eqref{Dqinv1} and \eqref{Dqinv2}, since $S$ is independent of $\cR $. For the polynomials $M_h=M_h(\cR )$ we have that
\begin{align}\label{Mmpro1}
M_h(\cR )&=-M_h(\cR_{i\leftrightarrow j}),\quad\mbox{if}\quad h\neq i,j\\
\nonumber M_i(\cR )&=-M_j(\cR_{i\leftrightarrow j}).
\end{align}
and
\begin{align}\label{Mmpro2}
M_h(\cR_{i\leftrightarrow ai+bj})&=aM_h(\cR ),\quad\mbox{if}\quad h\neq i,j\\
\nonumber M_i(\cR_{i\leftrightarrow ai+bj})&=M_i(\cR ),\\
\nonumber M_j(\cR_{i\leftrightarrow ai+bj})&=aM_j(\cR )-bM_i(\cR ).
\end{align}
These properties are a consequence of writing $M_h(\cR )$ as (see \eqref{emeiexp})
\begin{equation}\label{Mmh}
M_h(\cR )=\sum_{u=1}^m(-1)^{h+u}S(x+u)V_{h,u}(\cR _{\{ h\} }),
\end{equation}
where
$$
V_{h,u}(\cR _{\{ h\} })=\det\big(R_l(x-r)\big)_{\scriptsize \left\{\begin{array}{l}
                                                           l\in\{1,2,\ldots,m\}\setminus\{h\} \\
                                                            r\in\{-u+1,-u+2,\ldots,m-u\}\setminus\{0\}
                                                          \end{array}\right\}
}.
$$
Observe that $V_{h,u}$ is independent of the polynomial $R_h$. Now, it is straightforward to see from the definition of $V_{h,u}$, that
\begin{align*}
V_{h,u}(\cR_{\{ h\} })&=-V_{h,u}((\cR_{i\leftrightarrow j})_{\{ h\} }),\quad\mbox{if}\quad h\neq i,j\\
V_{i,u}(\cR_{\{ i\} })&=(-1)^{j-i-1}V_{j,u}((\cR_{i\leftrightarrow j})_{\{ j\} })\\
V_{j,u}(\cR_{\{ j\} })&=(-1)^{j-i-1}V_{i,u}((\cR_{i\leftrightarrow j})_{\{ i\} })
\end{align*}
and
\begin{align*}
V_{h,u}((\cR_{i\leftrightarrow ai+bj})_{\{ h\} })&=aV_{h,u}(\cR_{\{ h\} }),\quad\mbox{if}\quad h\neq i,j\\
V_{i,u}((\cR_{i\leftrightarrow ai+bj})_{\{ i\} })&=V_{i,u}(\cR_{\{ i\} }),\\
V_{j,u}((\cR_{i\leftrightarrow ai+bj})_{\{ j\} })&=aV_{j,u}(\cR_{\{ j\} })+b(-1)^{j-i-1}V_{i,u}(\cR_{\{ i\} }).
\end{align*}
Using the properties above in \eqref{Mmh} we obtain \eqref{Mmpro1} and \eqref{Mmpro2}.

From the definition of $D_{q,S}$ in (\ref{Dq}) and using the properties above for $P_S$ and $M_h, h=1,\ldots,m,$ it is now straightforward to see that \eqref{Dqinv1} and \eqref{Dqinv2} hold.

\end{proof}

We finally show the lemma that gives the order of the differential operator $D_{q,S}$ \eqref{Dq}.

\begin{lemma}\label{l3.5} In  Theorem \ref{Teor1}, assume in addition that $\A$ is the algebra of differential operators (\ref{algdiffc}),  the differential operator $D_p$ has order $r$ and the $\D$-operator $\D$ (which it is now also differential) has order $s$ with $s<r$. Then the order of the operator $D_{q,S}$ is $r(\deg S+\deg \Omega +1)$.
\end{lemma}

\begin{proof}
Indeed, the operator $D_{q,S}$ (\ref{Dq}) is the sum of the operators $T_1=P_S(D_{p})$ and $T_2=\sum_{h=1}^mM_h(D_{p})\D R_h(D_{p})$. Since the order of the differential operator $D_{p}$ is $r$, it is clear from the definition of the polynomial $P_S$ that the order of $P_S(D_{p})$ is just $r(\deg S+\deg \Omega +1)$. It is now enough to prove that the order of the operator $T_2$ is less than the order of $T_1$.

To stress the dependence of the polynomials $P_S$, $M_h$ (\ref{emeiexp}) and the operator $D_{q,S}$ on the $m$-tuple of polynomials
$\cR =(R_1,\ldots ,R_m)$, we write  $P_S=P_S(\cR )$, $M_h=M_h(\cR )$ and $D_{q,S}=D_{q,S}(\cR )$.
Using the invariance properties (\ref{Dqinv1}) and (\ref{Dqinv2}), we can get from the polynomials $R_i$, $i=1,\ldots , m$, new polynomials $\tilde R_i$, $i=1,\ldots , m$, satisfying that $\deg \tilde R_i\not =\deg \tilde R_{j}$, $i\not =j$ and
\begin{align*}
P_S&=P_S(\cR )=P_S(\tilde \cR ),\\
D_{q,S}&=D_{q,S}(\cR )=D_{q,S}(\tilde \cR ),
\end{align*}
where $\tilde \cR =(\tilde R_1,\ldots ,\tilde R_m)$.
If we write $\tilde M_h=M_h(\tilde \cR )$, $h=1,\ldots , m$, we then have
\begin{equation}\label{ot2}
T_2=\sum_{h=1}^m\tilde M_h(D_{p})\D \tilde R_h(D_{p}).
\end{equation}

Using \eqref{fL1}, we then get that the degree of the polynomial $P_S=P_S(\tilde \cR )$ is $\deg S+\left(\sum_{i=1}^m\deg \tilde R_i\right) -\binom{m}{2}+1$. This gives that the order of the operator $T_1$ is $r\deg S+r\left(\sum_{i=1}^m\deg \tilde R_i\right) -r\binom{m}{2}+r$.

A straightforward computation using (\ref{ot2}) shows that the order of the operator $T_2$ is less than or equal to
$$
\max \{ r\deg \tilde M_h+r\deg \tilde R_h +s, h=1,\ldots ,m \}.
$$
Using now \eqref{fL2}, we get that the degree of the polynomial $\tilde M_h$ is less than or equal to
$$
\deg S+\left(\sum_{i=1,i\not =h}^m\deg \tilde R_i\right) -\frac{m(m-1)}{2}.
$$
Hence for $h=1,\ldots , m$, we have
$$
\deg \tilde M_h+\deg \tilde R_h\le \left(\sum_{i=1}^m\deg \tilde R_i\right) -\frac{m(m-1)}{2}+\deg S.
$$
This gives that the order of the operator $T_2$ is less than or equal to
$r\deg S+r\left(\sum_{i=1}^m\deg \tilde R_i\right) -r\binom{m}{2}+s$. Since $s<r$, we conclude that the order of the differential operator $T_2$ is less than the order of $T_1$. This completes the proof of the Lemma.
\end{proof}

\section{Discrete Laguerre-Sobolev orthogonal polynomials}

In this Section we prove Theorems \ref{mainth1} and \ref{mainth2}
of the Introduction.

\begin{proof} of Theorem \ref{mainth1}

The Theorem is a direct consequence of Lemma \ref{LemLagSob}.
Indeed, we have just to identify who the main characters are in
this case. Since $\lambda=0$ and $\nu(x)=\mu_{\alpha-m}(x)$, in
our case, we have $x^m\nu=\mu_{\alpha}$ and hence
$p_n=L_n^{\alpha}$, where as above $\mu _\alpha$ denotes the orthogonalizing weight for the Laguerre polynomials $(L_n^{\alpha})_n$. Using that
$L_n^{\alpha}(0)=\binom{n+\alpha}{n}$ and \eqref{Lagder} we get
$\left(L_n^{\alpha}\right)^{(i)}(0)=(-1)^i\binom{n+\alpha}{n-i}$.
Also, using \eqref{Lagab} for $\beta=\alpha-m+l$ we get
$w_{n,l}=\frac{\Gamma (\alpha -m+l)}{(m-l)!}(n+1)_{m-l}$.  The
expressions (\ref{RdSip}) and (\ref{quss}) give then
(\ref{iRdSip}) and (\ref{iquss}), respectively.

\end{proof}

\begin{proof} of Theorem \ref{mainth2}

On the one hand, when $\alpha$ is a positive integer with $\alpha \ge
m$, the functions $\mathcal R_l$ in Theorem \ref{mainth1} are
actually polynomials (see \eqref{iRdSip}). On the other hand the $\D $-operator for the
Laguerre polynomials displayed in the Lemma \ref{lTlag} is defined
from the sequence of numbers $\varepsilon _n=-1$. Hence, the
expression (\ref{iquss}) for the orthogonal polynomials $(q_n)_n$
with respect to the discrete Laguerre Sobolev bilinear form
(\ref{idslip}) fits with the expression (\ref{qus}) in Theorem
\ref{Teor1} for $p_n=L_n^\alpha $, $\varepsilon _n=-1$. Hence we
deduce that for any polynomial $S$, there exists a differential
operator $D_S$  with respect to which the orthogonal polynomials
$(q_n)_n$ (\ref{iquss}) are eigenfunctions. This operator $D_S$ can be explicitly constructed using
(\ref{Dq}), where in this case $D_p$ is the second-order differential operator $D_{\alpha}$ for the Laguerre polynomials defined by \eqref{DopL}.
Moreover, up to an
additive constant, the corresponding eigenvalues $(\lambda_n)_n$
of $D_S$ are $\lambda _n=P_S(n)$, where $P_S$ is the polynomial
defined by the first order difference equation
$P_S(x)-P_S(x-1)=S(x)\Omega (x)$, and $\Omega$ is the Casorati
determinant (\ref{casdet}). Using Lemma \ref{l3.5}, we conclude
that the order of the differential operator $D_S$ is  equal to
$2(\deg \Omega +\deg S +1)$. To finish the proof it is enough to
use Lemma \ref{Lemawr}.

\end{proof}

\begin{lemma}\label{Lemawr}
For $m\ge 1$, let $M=(M_{i,j})_{i,j=0}^{m-1}$ be a $m\times m$
matrix. Let $\alpha$ be a positive integer with $\alpha \ge m$. For
$j=1,\ldots , m$, define the polynomials $\mathcal{R}_j$ by \eqref{iRdSip}.
Then the degree of the Casorati determinant $\Omega (x)$ (see
\eqref{casdet}) is the $\alpha$-weighted rank $\awr (M)$ of $M$
defined in Definition \ref{wrM}.
\end{lemma}

\begin{proof}
Observe that the functions $\mathcal{R}_j, j=1,\ldots , m,$ are indeed polynomials since $\alpha$ is a positive integer. We can not apply the first  part of Lemma \ref{lgp1} yet, because in general, some of the polynomials $\mathcal{R}_j$ can have equal degrees. To avoid this problem we will transform the polynomials $\mathcal{R}_j$ in a suitable way.

We first factorize the polynomials $\mathcal{R}_j$ in the form
\begin{equation}\label{fact1}
\mathcal{R}_j(x)=s_j(x)+(j-1)!(x+1)_\alpha r_j(x),
\end{equation}
where
\begin{align*}
r_j(x)&=\sum_{i=0}^{m-1}\frac{(-1)^iM_{j-1,i}}{(\alpha +i)!}(x-i+1)_i,\\
s_j(x)&=\frac{(\alpha -j+j-1)!}{(m-j)!}(x+1)_{m-j}.
\end{align*}
They satisfy that $\deg r_j\le m-1$, $\deg s_j= m-j\le m-1$, $j=1,\ldots , m$. Since $\alpha\ge m$, we also get $\deg ((x-\alpha)_\alpha r_j(x))\ge \deg s_j$.

From the definition of $r_j, j=1,\ldots,m,$ it is clear that
\begin{equation}\label{fact2}
\begin{pmatrix}
  r_1\\
  r_2\\
  \vdots \\
  r_m\\
\end{pmatrix}=MT_1\begin{pmatrix}
  1 \\
  x \\
  \vdots \\
  x^{m-1}\\
\end{pmatrix},
\end{equation}
where $T_1=(T_{1,i,j})_{i,j=0}^{m-1}$ is a nonsingular lower triangular matrix with entry $T_{1,i,j}$ equal to the coefficient of $x^{j}$ in the power expansion of $(-1)^i(x-i+1)_i/(\alpha +i)!$.

We next find a convenient factorization of $M$ which it will allow us to calculate the degrees of the polynomials $r_j, s_j, j=1,\ldots,m$, in terms of the $\alpha$-weighted rank of the matrix $M$. For that, write $X=\{j:n_j\not =0\}$, where $n_j$, $j=1,\ldots , m$, are the numbers defined by (\ref{nj}). For certain $s$, $1\le s\le m$, we can write $X=\{l_{m-s+1},\ldots , l_m\}$ with $l_i<l_{i+1}$, and then $n_{l_i}=\alpha+m-l_i$.

Applying the QL factorization (see \cite{HJ}, pp. 112-113) to $M$, we find $q_{m-s+1},\ldots ,q_m,$ unitary orthogonal vectors in $\RR ^m$, and vectors $v_{m-s+1},\ldots, v_m\in \RR ^m$ with
\begin{equation}\label{propv}
v_{m-i+1,m-l_{m-s+i}+1}\not =0,\quad v_{m-i+1,j}=0,
\end{equation}
for $j=m-l_{m-s+i}+2,\ldots , m$ and $i=1,\ldots ,s$, such that
\begin{equation}\label{qrtM}
M=(q_{m-s+1}^t \ldots q_m^t)\begin{pmatrix}v_{m-s+1}\\ \vdots \\ v_m\end{pmatrix}.
\end{equation}
In addition to the numbers $m_j, j=1,\ldots,m-1$, defined by \eqref{mj}, we also define the number $m_m$ by $m_m=0$ if $f_m\not=0$ and $m_m=1$ if $f_m=0$ (we are using here the notation of Definition \ref{wrM}). Write now $Y=\{j:m_j =0\}$ and $Z=\{j:m_j \not =0\}$. Since $M$ and $\tilde M$ have equal rank (see Definition \ref{wrM}), we have $|X| =|Y|$. We can then write
$Y=\{g_{m-s+1},\ldots , g_m\}$ and $Z=\{h_{1},\ldots , h_{m-s}\}$, with $g_i<g_{i+1}$,
$h_i<h_{i+1}$ and $m_{h_i}=m-h_i$. As a consequence of the definition of the numbers $m_j$,  we have
$f_{h_l}\in \langle f_{h_l+1},\ldots , f_m\rangle$ and $\langle f_{h_l+1},\ldots , f_m\rangle=\langle f_u: u\in Y, u>h_l\rangle$; hence
\begin{equation}\label{partt2}
\mbox{$f_{h_l}\in \langle f_u: u\in Y, u>h_l\rangle$, \quad $l=1,\ldots, m-s$.}
\end{equation}
We finally remark that
\begin{align}\label{partt1}
&\mbox{the finite sets $Y=\{g_{m-s+1},\ldots , g_m\}$ and $Z=\{h_{1},\ldots , h_{m-s}\}$} \\ \nonumber &\mbox{form a partition of $\{1,2,\ldots, m\}$.}
\end{align}
We now assume the following claim.

\noindent
\textsl{Claim.} There exist unitary orthogonal vectors $q_1,\ldots, q_{m-s}$ such that:
\begin{enumerate}
\item $q_1,\ldots, q_{m}$ form a unitary orthogonal basis of $\RR ^m$;
\item for  $\mu$, $1\le \mu \le m-s$, and $\nu$, $1\le \nu <h_\mu$, we have that
\begin{equation}\label{en1}
q_{\mu,h_\mu}\not =0\quad \mbox{and} \quad q_{\mu,\nu}=0.
\end{equation}
\end{enumerate}

Write now  $U$ for the $m\times m$  unitary matrix whose columns are $q_i^t$, $i=1,\ldots , m$. In addition,
write $T$ for the $m\times m$  matrix whose rows are $0$ for $i=1,\ldots , m-s$ and $v_i$ for $i=m-s+1,\ldots , m,$ (the $v$'s are the vectors from the QL factorization of the matrix $M$ (\ref{qrtM})).
Using (\ref{qrtM}), we get
\begin{equation}\label{mut}
M=UT.
\end{equation}
For $i=1,\ldots , m$, define now the polynomials
$$
\mathcal{\tilde R}_i=q_i\begin{pmatrix}\mathcal{R}_1\\ \vdots \\ \mathcal{R}_m \end{pmatrix},\quad \tilde r_i=q_i\begin{pmatrix}r_1\\ \vdots \\ r_m \end{pmatrix},\quad
 \tilde s_i=q_i\begin{pmatrix}s_1\\ \vdots \\ s_m \end{pmatrix}.
$$
The factorization (\ref{fact1}) shows that $\mathcal{\tilde R}_i(x)=\tilde s_i (x)+(i-1)!(x+1)_\alpha \tilde r_i(x)$, $\deg \tilde r_i,\deg \tilde s_i\le m-1$. Since the matrix $U$ is unitary, (\ref{fact2}) and (\ref{mut}) give that
$$
\begin{pmatrix}
  \tilde r_1\\
 \tilde r_2\\
  \vdots \\
 \tilde r_m\\
\end{pmatrix}=TT_1\begin{pmatrix}
  1 \\
  x\\
  \vdots \\
  x^{m-1} \\
\end{pmatrix}.
$$
Since $T_1$ is an nonsingular lower triangular matrix and taking into account (\ref{propv}), we deduce
$$
 \tilde r_i=0,\quad i=1,\ldots , m-s, \quad \deg \tilde r_{m-i+1}= m-l_{m-s+i},\quad i=1,\ldots, s.
$$
On the other hand, since $\deg s_i=m-i$ and taking into account (\ref{en1}) and the definition of $m_m$, we deduce
$\deg \tilde s_i=m-h_i$, $i=1,\ldots , m-s$, except if $h_{m-s}=m$ in which case $\deg \tilde s_i=m-h_i$, $i=1,\ldots , m-s-1$,
$\deg \tilde s_{m-s}=0$. Therefore, if $h_{m-s}<m$ then
$$
\deg \mathcal{\tilde R}_i=\deg \tilde s_i=m-h_i=m_{h_i},\quad i=1,\ldots,m-s,
$$
and if $h_{m-s}=m$ then
$$
\deg \mathcal{\tilde R}_i=\deg \tilde s_i=m-h_i=m_{h_i},\quad i=1,\ldots,m-s-1,\quad \deg \mathcal{\tilde R}_{m-s}=0.
$$
On the other hand, since $\deg \tilde s_i\le m-1$, $1\le i\le m$, $\alpha \ge m$ and $\tilde r_{m-i+1}\not =0$, $i=1,\ldots,s$, we get
$$
\deg \mathcal{\tilde R}_{m-i+1}=\alpha+\deg \tilde r_{m-i+1} =\alpha +m-l_{m-s+i}=n_{l_{m-s+i}},\quad i=1,\ldots,s.
$$
This implies that $\deg \mathcal{\tilde R}_i\not =\deg \mathcal{\tilde R}_j$, $i\not =j$. We can then apply Lemma \ref{lgp1} to get that
$$
\deg \left(\det \left(( \mathcal{\tilde R}_i(x-j))_{i,j=1}^m\right)\right)=\sum _{j=1}^m n_j+\sum _{j=1}^{m-1}m_j-\frac{m(m-1)}{2}.
$$
Since $(\mathcal{\tilde R}_i(x-j))_{i,j=1}^m=U^*\left( \mathcal{R}_i(x-j)\right)_{i,j=1}^m$ and $U$ is unitary, we finally get
$$
\deg \left(\det \left((\mathcal{R}_i(x-j))_{i,j=1}^m\right)\right)=\sum _{j=1}^m n_j+\sum _{j=1}^{m-1}m_j-\frac{m(m-1)}{2}=\awr(M).
$$
We finally prove the Claim. We proceed in three steps.

\bigskip

\noindent
\textsl{First step}. Denote by $e_i$, $i=1,\ldots , m$, the canonical basis in $\RR ^m$, that is $e_{i,j}=\delta_{i,j}$. Then the vectors $e_{h_1},\ldots ,e_{h_{m-s}}, q_{m-s+1}, \ldots , q_m$ are linearly independent.

Consider again the set of numbers $X=\{l_{m-s+1},\ldots , l_m\}$ with $n_{l_k}\not =0$, $k=m-s+1,\ldots, m$ and $l_i<l_{i+1}$,
$Y=\{g_{m-s+1},\ldots , g_m\}$ and $Z=\{h_{1},\ldots , h_{m-s}\}$, with $m_{g_l}=0$, $l=m-s+1,\ldots, m$, $m_{h_l}\not =0$, $l=1,\ldots , m-s$, $g_i<g_{i+1}$ and $h_i<h_{i+1}$.
The definition of the numbers $n_j$, $m_j$ (\ref{nj}), (\ref{mj}) implies that
\begin{equation}\label{notz}
\det (M_{g_i,l_j})_{i,j=m-s+1,\ldots, m}\not =0.
\end{equation}
The factorization (\ref{qrtM}) gives
\begin{equation}\label{4.4}
M_{g_i,l_j}=(q_{m-s+1,g_i},\ldots, q_{m,g_i})\begin{pmatrix}v_{m-s+1,l_j}\\ \vdots \\ v_m,l_j\end{pmatrix},
\end{equation}
or equivalently
$$
(M_{g_i,l_j})_{i,j=m-s+1,\ldots, m}=(q_{i,g_j})_{i,j=m-s+1,\ldots, m}(v_{i,l_j})_{i,j=m-s+1,\ldots, m}.
$$
From (\ref{notz}) we deduce
\begin{align}\label{li1}
\det (q_{i,g_j})_{i,j=m-s+1,\ldots, m}&\not=0,\\ \label{li2}
\det (v_{i,l_j})_{i,j=m-s+1,\ldots, m}&\not=0.
\end{align}
Taking into account (\ref{partt1}), we get from (\ref{li1})
$$
\det (e_{h_1}^t,\ldots ,e_{h_{m-s}}^t, q_{m-s+1}^t, \ldots , q_m^t)=\epsilon\det (q_{i,g_j})_{i,j=m-s+1,\ldots, m}\not =0,
$$
where $\epsilon$ is just a sign: $\epsilon=\pm 1$. That is, the vectors $e_{h_1},\ldots ,e_{h_{m-s}}, q_{m-s+1}, \ldots , q_m$ are linearly independent.

\bigskip

\noindent
\textsl{Second step}. There exist unitary orthogonal vectors $q_1,\ldots, q_{m-s}$ such that:
\begin{enumerate}
\item $q_1,\ldots, q_{m}$ form a unitary orthogonal basis of $\RR ^m$,
\item for  $\mu$, $1\le \mu \le m-s$, and $\nu$, $1\le \nu <h_\mu$, we have that
\begin{equation}\label{en2}
e_\nu \in \langle q_1,\ldots,q_{\mu-1}, q_{m-s+1},\ldots, q_m\rangle ,
\end{equation}
\begin{equation}\label{en3}
e_{h_\mu} \not \in \langle q_1,\ldots, q_{\mu-1}, q_{m-s+1},\ldots, q_m\rangle ,
\end{equation}
\begin{equation}\label{en4}
e_{h_\mu}  \in \langle q_1,\ldots, q_{\mu-1}, q_\mu, q_{m-s+1},\ldots, q_m\rangle ,
\end{equation}
\begin{equation}\label{en5}
q_\mu \in \langle q_1,\ldots, q_{\mu-1}, e_{h_\mu}, q_{m-s+1},\ldots, q_m\rangle .
\end{equation}
\end{enumerate}

Consider first $\mu =1$. We now see that for $1\le \nu <h_1$, $e_\nu \in \langle q_{m-s+1},\ldots, q_m\rangle$ (that is, (\ref{en2}) holds). This is equivalent to prove that
$\rank (e_\nu ^t,q_{m-s+1}^t, \ldots , q_m^t)=s$. Taking into account (\ref{partt1}), it is enough to prove it for the indexes in $Y$ such that $1\leq g_{m-s+\nu}<h_1$. Using (\ref{li1}) and again (\ref{partt1}), it will be enough to prove that for  $l=1,\ldots , m-s$,
$$
\det \begin{pmatrix}0&q_{m-s+1,g_{m-s+1}}&\ldots &q_{m,g_{m-s+1}}\\
\vdots&\vdots&\ddots&\vdots\\
1&q_{m-s+1,g_{m-s+\nu}}&\ldots &q_{m,g_{m-s+\nu}}\\
\vdots&\vdots&\ddots&\vdots\\
0&q_{m-s+1,g_{m}}&\ldots &q_{m,g_{m}}\\
0&q_{m-s+1,h_l}&\ldots &q_{m,h_{l}}\\
\end{pmatrix}=0.
$$
To simplify the writing, we set $W=\{m-s+1,\ldots ,m\}$,
$Y_\nu=Y\setminus \{g_{m-s+\nu}\}$ and $Y^l_\nu =Y_\nu \cup
\{h_l\}$. Hence, we have to prove that
$$
\det (q_{i,j})_{i\in W,j\in Y^l_\nu}=0.
$$
Using (\ref{4.4}), we get
$$
(M_{i,l_j})_{i\in Y^l_\nu ,j\in W}=(q_{i,j})_{i\in W,j\in Y^l_\nu } (v_{i,l_j})_{i,j\in W}.
$$
With the notation of Definition \ref{wrM}, the rows of the matrix $(M_{i,l_j})_{i\in Y^l_\nu,j\in W}$ are (from up to down) $f_i$, $i\in Y^l_\nu$.
Using (\ref{partt2}), we get $f_{h_l}\in \langle f_u: u\in Y, u>h_l \rangle $. But $g_{m-s+\nu}<h_1\le h_l$, hence
$\{u:u\in Y,u>h_l\}\subset Y_\nu$. This gives $f_{h_l}\in \langle f_{u}, u\in Y_\nu \rangle $.
That is, the last row of the matrix $(M_{i,l_j})_{i\in Y^l_\nu ,j\in W}$ is a linear combination of the other rows. This gives $\det (M_{i,l_j})_{i\in Y^l_\nu ,j\in W}=0$. Taking into account (\ref{li2}), we finally get $\det (q_{i,j})_{i\in W,j\in Y^l_\nu }=0$. That is for $1\le \nu <h_1$, $e_\nu \in \langle q_{m-s+1},\ldots, q_m\rangle$.

On the other hand, using the first step we get that $e_{h_1}\not  \in \langle q_{m-s+1},\ldots, q_m\rangle$ (that is, (\ref{en3}) holds).

Using Lemma \ref{cdm}, we get a unitary vector $q_{1}$ satisfying that:
\begin{enumerate}
\item $q_{1},q_{m-s+1},\ldots, q_m$ are unitary orthogonal vectors,
\item $e_{h_1}  \in \langle q_1, q_{m-s+1},\ldots, q_m\rangle $ (therefore we get \eqref{en4}).
\item $q_1\in \langle e_{h_1}, q_{m-s+1},\ldots, q_m\rangle $ (therefore we get \eqref{en5}).
\end{enumerate}
This proves the second step for $\mu =1$.

The vectors $q_2,\ldots , q_{m-s}$ can be constructed proceeding in a similar way.

\noindent
\textsl{Third step}. The unitary orthogonal vectors $q_1,\ldots , q_{m}$ satisfy
that for  $\mu$, $1\le \mu \le m-s$, and $\nu$, $1\le \nu <h_\mu$,
$q_{\mu,h_\mu}\not =0$ and $q_{\mu,\nu}=0$.

Indeed, since the unitary vectors $q_1,\ldots, q_{m}$ are orthogonal, (\ref{en2}) of the second step  implies that $q_{\mu,\nu}=0$. On the other hand,
(\ref{en3}) and (\ref{en4}) of the second step  show that $q_{\mu,h_\mu}\not =0$.

\end{proof}

Some interesting particular cases are the following:
\begin{enumerate}
\item When $M=(M_{i,j})_{i,j=0}^{m-1}$ is the symmetric matrix with entries
$$
M_{i,j}=\begin{cases} a_{i+j},& i+j\le m-1,\\ 0,& i+j> m-1,\end{cases}
$$
($a_{m-1}\neq0$) the discrete Laguerre Sobolev inner product (\ref{idslip}) reduces to the inner product defined by the moment functional
$$
x^{\alpha-m}e^{-x}+\sum_{i=0}^{m-1}a_i\delta_0^{(i)},
$$
introduced by Gr\"unbaum, Haine and Horozov, in \cite{GrHH}. In this case it is easy to see that $\deg \mathcal{R}_j=\alpha +m-j$ and then using Lemma \ref{lgp1}, we have $\deg \Omega =m\alpha $. Our result improves the order of the operator given in  \cite{GrHH} and agrees with the order found by Iliev in \cite{I} for this kind of inner product.
In particular, for $a_1=a_2=\cdots =a_{m-1}=0$, we get orthogonal polynomials with respect to the measure
$$
x^{\alpha -1}e^{-x}+a_0\delta _0,\quad x>0,
$$
and we recover Koekoeks' result (\cite{koekoe}).
\item When $M$ is diagonal, $M=\diagonal (M_0,\ldots , M_{m-1})$, $M_{m-1}\not =0$, we get for $j=1,\ldots , m$, that
$$
\deg \mathcal{R}_j=\begin{cases} \alpha +j-1,& M_{j-1}\not =0,\\ m-j,& M_{j-1}=0.\end{cases}
$$
Using again Lemma \ref{lgp1}, we have
$$
\deg \Omega =s\alpha +(m-s)(m+1)-2\sum_{j:1\le j\le m, M_{j-1}=0}j,\quad s=|\{j:1\le j\le m, M_j\not =0\}|.
$$
For $m=2$, we  recover Koekoeks' and Bavinck results (\cite{KKB}).
\end{enumerate}

\bigskip

We finish this paper illustrating the case when $\alpha=3$, $m=3$ and
$$
M=\begin{pmatrix} 1&1&0\\1&1&0\\0&0&1\end{pmatrix}.
$$
In this case, the polynomials $\mathcal R_j, j=1,2,3$ (see \eqref{iRdSip}) are given by
\begin{eqnarray*}
  \mathcal R_1(x) &=& -\frac{(x+1)(x+2)(x^2-x-24)}{24}, \\
  \mathcal R_2(x) &=& -\frac{(x+1)(x^3+x^2-14x-48)}{24}, \\
  \mathcal R_3(x) &=& \frac{(x+4)(x^4+x^3+x^2-9x+30)}{60}.
\end{eqnarray*}
The Casorati determinant $\Omega$ (see \eqref{casdet}) is given by
$$
\Omega(x)=-\frac{x^8}{480}+\frac{x^7}{40}-\frac{91x^6}{720}+\frac{3x^5}{20}+\frac{613x^4}{1440}-\frac{71x^3}{40}+\frac{1333x^2}{360}-\frac{22x}{5}-2.
$$
Since for this particular matrix $M$, the bilinear form (\ref{idslip}) is actually an inner product, there exist a sequence $(q_n)_n$ of orthogonal polynomials with respect to it, which can be constructed using (\ref{iquss}).

The solution of \eqref{Pgs} (for $S(x)=1$) is given by
$$
P_{S}(x)=-\frac{x^9}{4320}+\frac{x^8}{480}-\frac{x^7}{144}-\frac{17x^6}{720}+\frac{47x^5}{480}-\frac{253x^4}{1440}+\frac{55x^3}{108}-\frac{289x^2}{360}-\frac{18x}{5}.
$$
Finally, the polynomials $M_j, j=1,2,3$ (see \eqref{emeiexp}) are given by
\begin{eqnarray*}
  M_1(x) &=& -\frac{11x^6}{120}-\frac{11x^5}{40}+\frac{29x^4}{24}+\frac{87x^3}{8}+\frac{1553x^2}{60}+\frac{152x}{5}+12, \\
  M_2(x) &=& \frac{11x^6}{120}+\frac{11x^5}{40}-\frac{71x^4}{24}-\frac{131x^3}{8}-\frac{512x^2}{15}-\frac{369x}{10}-14, \\
  M_3(x) &=& -2x^3-3x^2-3x-7.
\end{eqnarray*}
The order of the differential operator \eqref{Dq} is 18, which is exactly the value of $2(\awr (M)+1)$ (here $n_1=5, n_2=4, n_3=0$ and $m_1=2, m_2=0$). This order is also given by $2\deg P_{S}(x)$. Observe that the order of $M_1(D_{\alpha})\frac{d}{dx}\mathcal{R}_1(D_{\alpha})$ and $M_2(D_{\alpha})\frac{d}{dx}\mathcal{R}_2(D_{\alpha})$ is 21, but there is no contradiction with the results of this paper since the order of $\sum_{h=1}^3M_h(D_{\alpha})\frac{d}{dx}\mathcal{R}_h(D_{\alpha})$ is less than or equal to 18, as a consequence of canceling the  two first coefficients of $M_1$ and $M_2$, and the first coefficient of $\mathcal{R}_1$ and $\mathcal{R}_2$.

\end{document}